\documentclass[11pt, a4paper]{article}

\usepackage{amsmath,amssymb,amsthm}
\usepackage{url,eurosym}
\usepackage{graphicx,color} 
\usepackage{subcaption}
\usepackage{tabularx}

\newtheorem{theorem}{Theorem}[section]
\newtheorem{lemma}[theorem]{Lemma}

\newtheorem{proposition}[theorem]{Proposition}

\newcommand{\rr}{\mathbb{R}}
\newcommand{\st}{\mathop{\rm s.t.}\nolimits}

\newcommand{\beqar}{\begin{eqnarray}}
\newcommand{\eeqar}{\end{eqnarray}}
\newcommand{\beqarno}{\begin{eqnarray*}}
\newcommand{\eeqarno}{\end{eqnarray*}}
\newcommand{\ba}[1]{\begin{array}{#1}}
\newcommand{\ea}{\end{array}}

\usepackage[T1]{fontenc}
\usepackage[utf8]{inputenc} % or \usepackage[utf8x]{inputenc} for more characters
\usepackage{graphicx,color} % for pdf, bitmapped graphics files
\usepackage{epsfig} % for postscript graphics files
\usepackage{times} % assumes new font selection scheme installed
\usepackage{amsmath} % assumes amsmath package installed
\usepackage{amssymb}  % assumes amsmath package installed
\usepackage{caption}
\usepackage{subcaption}
\usepackage{stackrel}

\usepackage{url}
\usepackage{multirow}
\usepackage{cite}

\usepackage{algorithm}
\usepackage{pgf}
\usepackage{pgfplots}
\usepgfplotslibrary{fillbetween}
\usepackage{tikz}
\usetikzlibrary{arrows,automata}
\usetikzlibrary{intersections}
\usetikzlibrary{calc}
\usetikzlibrary{shapes}
\usetikzlibrary{positioning}
\usetikzlibrary{decorations.text}
\pgfdeclarelayer{background}
\pgfdeclarelayer{foreground}
\pgfsetlayers{background,main,foreground}
\usetikzlibrary{calc,patterns,decorations.pathmorphing,decorations.markings}

\definecolor{mycol}{RGB}{19,48,128}
\usepackage{hyperref} 
\hypersetup{
	colorlinks = true,
	citecolor = mycol,
	linkcolor = mycol,
	urlcolor = black
}

\usepackage{enumitem}

\newcommand{\DD}{{\mathcal D}}
\newcommand{\PP}{{\mathcal P}}
\newcommand{\KK}{{\mathcal K}}

\newcommand{\XX}{{\mathcal X}}

\newtheorem{assumption}{Assumption}

\begin{document}

	\title{Learning Approximate Solutions to Multiparametric Generalized Nash Equilibrium Problems}	

    \author{Alberto Bemporad and Tatiana Tatarenko 
    {\renewcommand{\thefootnote}{}\thanks{A. Bemporad is with the IMT School for Advanced Studies, Piazza San Francesco 19, Lucca, Italy. Email: \texttt{\scriptsize alberto.bemporad@imtlucca.it}. T. Tatarenko is with the Department of Control Theory and Intelligent Systems, TU Darmstadt, Germany. E-mails: \texttt{\scriptsize tatiana.tatarenko@tu-darmstadt.de}.   
    This work was funded by the European Union (ERC Advanced Research Grant COMPACT, No. 101141351). Views and opinions expressed are however those of the authors only and do not necessarily reflect those of the European Union or the European Research Council. Neither the European Union nor the granting authority can be held responsible for them. The work was also funded by the Deutsche Forschungsgemeinschaft (DFG, German Research Foundation). Project number: 528033031.
    }}}
	
\maketitle
\thispagestyle{empty}
	
%%%%%%%%%%%%%%%%%%%%%%%%%%%%%%%%%%%%%%%%%%%%%%%%%%%%%%%%%%%%%%%%%%%%%%%%%%%%%%%%
\begin{abstract}
We propose a learning-based approach for approximating solution mappings of multiparametric generalized Nash equilibrium problems (GNEPs) with coupling in both objectives and constraints. Rather than solving a standard regression problem on a training dataset of GNEP solutions, which are expensive and possibly difficult to collect, we use the Nikaido-Isoda (NI) gap function as a training loss, which requires only best-response data. To avoid bilevel optimization, a value-function surrogate approximates each agent's optimal best-response cost and is substituted into the NI loss, yielding a single-level learning problem. Learning approximate solutions to standard multiparametric programming problems is a special case of the approach. We also establish new sufficient conditions for the existence of a continuous parametric variational GNE selection under a strong variational stability assumption that generalizes strong monotonicity. The trained neural network delivers approximate GNE solutions with speedups of several orders of magnitude over online solvers. Numerical experiments on different problem classes confirm the effectiveness of the approach. A Python library is available at \url{https://github.com/bemporad/mpfit}.
\end{abstract}

\textbf{Keywords:} Generalized Nash equilibrium problems, multiparametric programming, game-theoretic optimization, learning-based optimization.

\section{Introduction}\label{sec:introduction}
Research into the theoretical foundations of parametric convex optimization problems began in the 1960s~\cite{MR64} and saw significant development during the 1980s~\cite{Fia83}. In the 1970s, initial efforts focused on deriving explicit solutions to multiparametric linear programs (mpLPs)~\cite{GN72}. In the early 2000s, motivated by the need for explicit control laws in model predictive control, multiparametric quadratic programming (mpQP) was extensively investigated~\cite{BMDP02a}, and efficient algorithms were developed for solving mpQP~\cite{TJB03,PS10,GBN11} as well as mpLP~\cite{BBM03} problems.

Motivated by these developments, we next turn to an analogous research direction in multi-agent optimization with strategically interacting agents, namely, game-theoretic optimization. This area has gained significant attention due to the increasing prevalence of large-scale interconnected systems in which multiple decision-makers interact strategically, such as energy networks, communication systems, and economic markets. In such settings, the behavior of each agent both influences and is influenced by the actions of others, making game-theoretic models a natural framework for analysis and design.

At the same time, computing equilibria in such problems remains computationally demanding, particularly in large-scale or real-time settings. This challenge is especially pronounced in game-theoretic receding horizon and model predictive control (MPC) frameworks, where a generalized Nash equilibrium problem (GNEP) that depends on a vector of parameters, such as states and reference signals that change over time, must be solved at each sampling instant~\cite{HBLD22,BG25,LK24,zhu2023sqp_gne,Bem26}. While such problems can often be formulated as parametric equilibrium problems, solving them online may be prohibitive due to stringent time constraints and the possible absence of some structural assumptions, such as monotonicity and smoothness properties of the game, under which distributed optimization methods are provably efficient~\cite{BIANCHI2022, TatKamTAC19, TatNedShi21}. This motivates the development of parametric representations and the corresponding learning-based approaches that approximate solution mappings offline and enable fast online evaluation.

While parametric solution mappings are well understood in single-agent convex optimization, and sensitivity analysis provides local regularity results for equilibrium problems, the explicit representation and learning of solution mappings for generalized Nash equilibrium problems remain largely unexplored. 
An exception is the
case of linear-quadratic (LQ) games, where the cost functions are convex and quadratic and the constraints are linear, for which explicit LQ-GNE maps defined over polyhedral regions in the parameter space can be derived~\cite{HB26}.
A drawback of such exact characterizations is that they are restricted to small-scale problems, due to the rapid exponential growth of the number of regions of the solution map with the number of decision variables and parameters.

%\subsection{Contribution}
In this paper, we address the problem of approximating solution mappings of a rather general class of multiparametric GNEPs with coupling in both agents' objective functions and constraints. To be able to address the problem in such a very general setting, 
due to the possible absence of an exact explicit characterization of the solution, we formulate the GNEP as the minimization of the associated Nikaido-Isoda (NI) gap function. The NI function is employed as a loss function in a multivariate regression problem, which is addressed by training a feedforward neural network that approximates a parametric GNEP solution mapping. 

Our approach differs from standard function regression, which would require a dataset of GNEP solutions, each computed for a different fixed parameter value, which would be expensive and possibly difficult to collect. Instead, we train a solution model using only local best responses and cost function evaluations. 
Our approach can also be applied to find approximate solutions to standard multiparametric optimization problems, by simply considering a single agent or, equivalently, $N$ agents with decoupled cost functions and no shared constraints; in such a case, 
which involves no game, preparing the dataset has zero cost, as no optimization is required. 

The paper is organized as follows. In Section~\ref{sec:problem_formulation} we state the parametric GNE problem we consider. Section~\ref{sec:learning_problem} provides details on modeling the equilibrium solution map, constructing the training set, and defining the NI-based loss function for the learning process. Section~\ref{sec:structured} clarifies how the sensitivity analysis of parametric solutions in a structured class of games motivates the chosen continuous models for both the value-function and solution models. Section~\ref{sec:single-agent} discusses the special case of standard (single-agent) multiparametric optimization problems. Section~\ref{sec:examples} presents numerical experiments on linear-quadratic, quadratically-constrained quadratic, and nonlinear GNEPs, as well as standard multiparametric programming problems. Section~\ref{sec:conclusions} concludes the paper.

\subsection{Notation}
We denote by $\rr^n$ the $n$-dimensional Euclidean space and by $\rr^{m\times n}$ the set of real $m\times n$ matrices. For a vector $v\in\rr^n$, $\|v\|$ denotes the Euclidean norm and $\|v\|_1$ the $\ell_1$ norm. Given $u,v\in\rr^n$, $\langle u,v\rangle = u^\top v$ denotes their inner product. We use the subscript $_k$ to denote the $k$-th element of a set, $_i$ to denote objects related to the $i$-th agent, and $_{-i}$ to denote those related to all agents except agent $i$.

\section{Parametric Generalized Nash Equilibrium Problem}
\label{sec:problem_formulation}
We consider a {\it parametric} generalized Nash equilibrium problem involving $N$ agents (players), each minimizing an individual cost function. Let $x_i\in\rr^{n_i}$ denote the decision variable of agent $i$, with $\sum_{i=1}^N n_i = n_x$, and let
\[
J_i(x_i,x_{-i},p): \rr^{n_x}\times \rr^{n_p}\to \rr
\]
denote the cost function of agent $i$, where $p\in\PP\subseteq\rr^{n_p}$ is a parameter vector and $x_{-i}\in\rr^{n_x-n_i}$ collects the remaining decision variables.

For each agent $i$, we define the {\it best-response} map $\bar x_i(x_{-i},p)$ as a solution to the parametric optimization problem
\begin{subequations}
\begin{equation}
\ba{rcrl}
\bar x_i(x_{-i},p) &\in& \arg\min_{x_i} & J_i(x_i,x_{-i},p) \\
&& \st & g(x,p)\leq 0\\
&&& h(x,p)=0.
\label{eq:best_response}
\ea
\end{equation}
where $g:\rr^{n_x}\times\rr^{n_p}\to\rr^{n_g}$ and $h:\rr^{n_x}\times\rr^{n_p}\to\rr^{n_h}$ describe the coupling constraints {\it shared} by all agents, $x=(x_i,x_{-i})$ denotes the joint decision vector, and  $\bar J(x_{-i},p)$ is the value function of the best-response problem. We denote by
\begin{equation}
  \label{eq:value-function}
  \bar J_i(x_{-i},p)=J_i(\bar x_i(x_{-i},p),x_{-i},p)  
\end{equation} 
the value function of the best-response problem~\eqref{eq:best_response} for agent $i$.
Any local constraints of the form $x_i\in\XX_i(p)$, such as bounds on the individual decisions, are assumed to be incorporated into the constraints $g(x,p)\leq 0$ and $h(x,p)=0$. 
Note that, while each agent's problem~\eqref{eq:best_response} may depend only on a subvector $p_i$ of parameters in $p$, for maximum generality we consider $p_i=p$, $\forall i=1,\ldots,N$.

We assume a full-information setting in which all functions and constraint sets involved in~\eqref{eq:best_response} are known to a central entity responsible for solving the generalized Nash equilibrium problem. Thus, its objective is to determine a joint strategy $x^\star(p)=(x_1^\star(p),\ldots,x_N^\star(p))$ such that, for every $p\in\PP$,
\begin{equation}
x_i^\star(p)=\bar x_i(x_{-i}^\star(p),p), \qquad \forall i=1,\ldots,N.
\label{eq:GNE}
\end{equation}
\label{eq:GNEP-parametric}%
\end{subequations}
Under standard convexity and regularity assumptions, for instance, convexity and differentiability of $J_i$ with respect to $x_i$, convexity of the constraint set, and suitable constraint qualification conditions, the above problem can be reformulated as a quasi-variational inequality~\cite{facchinei2007generalized}. In this setting, several sufficient conditions for the existence of a solution are available. However, provably efficient numerical methods typically require additional structure, most notably monotonicity of the game which, together with convexity of shared constraints, allows for focusing on a subclass of solutions characterized by a specifically defined variational inequality~\cite{facchinei2007generalized}.

In this paper, our objective is to compute, for every $p\in\PP$, an approximate equilibrium $\hat x(p)\approx x^\star(p)$ by means of a suitably designed machine learning procedure, without imposing convexity or monotonicity assumptions on the underlying parametric game. 

\section{Learning Problem}
\label{sec:learning_problem}
The posed learning problem~\eqref{eq:GNEP-parametric} could be addressed as a standard multivariate function regression problem, such as training a feedforward neural network $\hat x$ on GNE data. However, this approach has two main drawbacks: first, a large number of GNEs must be computed for different values of $p$ to get a sufficiently large training dataset, which can be computationally expensive; second, multiple GNEs may result for the same $p$, and it is not obvious which one to pick to create the dataset,
unless one restricts to very specific GNE problems, where uniqueness of solution can be guaranteed. 
To overcome these issues, as well as to take into account that $\hat x(p)$ must represent a GNE solution, we will leverage the idea to learn solutions by means of the Nikaido-Isoda gap function~\cite{NikaidoIsoda1955} based on best-response data, which are much cheaper to compute.
It is well known and  straightforward to prove that solutions $x^\star(p)$ to the GNEP~\eqref{eq:GNEP-parametric} minimize the following non-negative Nikaido-Isoda (NI) gap function
\begin{equation}
    NI(x,p) = \sum_{i=1}^N\left( J_i(x_i,x_{-i},p) - J_i(\bar x_i(x_{-i},p),x_{-i},p)\right)
    \label{eq:NI-function}
\end{equation}
over the shared constraint set $\{x\,:\, g(x,p)\leq 0, \, h(x,p)= 0\}$.
Moreover, any minimizer, i.e., a zero point, of the function above is a GNE $x^\star(p)$.
In other words, $x^\star(p)$ solves the problem~\eqref{eq:GNEP-parametric} if and only if $NI(x^\star(p),p)=0$, along with $g(x^\star(p),p)\leq 0$ and $h(x^\star(p),p)=0$.

\subsection{Data sets and parametric models}
We define the training dataset by the collection of $M$ tuples 
\begin{equation}
    \DD=\{(x_k,p_k,\bar J_{k})\}, \,\, k=1,\ldots,M
    \label{eq:dataset}
\end{equation}
where $\bar J_{k,i}=\bar J_i(x_{-i,k},p_k)$ is the value-function of the 
best-response problem~\eqref{eq:best_response} for $x_{-i}=x_{-i,k}$ and $p=p_k$,
$\bar J_k\in\rr^N$.
We assume that $\bar J_{k}$ is defined for all $k$; this is not a restriction, as we simply exclude values $p_k$ from the dataset for which the feasible set $\{x\in\rr^{n_x}:\ g(x,p_k)\leq 0, \, h(x,p_k)= 0\}$ is empty or an agent's problem is unbounded.

To approach the learning problem, we consider two different parametric models:
\begin{enumerate}
\item the {\it value-function model} $\hat J_i(x_{-i},p,\theta_{1i})$, 
$\hat J_i:\rr^{n_x-n_i}\times\PP\to\rr$ for any $\theta_{1i}\in\rr^{n_{\theta_{1i}}}$, where  $\theta_{1i}$ is the model parameter vector to be defined;
\item the {\it GNE model} $\hat x(p,\theta_2)$, where $\theta_2$ is a vector of model parameters to be defined and $\hat x:\PP\to\rr^{n_x}$ for any $\theta_2\in\rr^{n_{\theta_2}}$.
\end{enumerate}
In this paper, we will consider feedforward neural networks as parametric models, but other classes of models can be considered as well. The choice of the class of models should be guided by the properties of the parametric GNE solution, such as continuity, smoothness, etc. More details on the model choice in games possessing a specific structure are provided in Section~\ref{sec:structured}.

\subsection{Optimization problem}
To solve the parametric GNEP~\eqref{eq:GNEP-parametric}
via the NI function~\eqref{eq:NI-function}, consider the following problem:
\begin{equation}
    \begin{aligned}
    \min_{\theta_2}\ &\frac{1}{M}\sum_{k=1}^M \sum_{i=1}^N \nu_i(p_k,\theta_2) \\ 
    \st~& g(\hat x(p_k,\theta_2),p_k)\leq 0\\
    & h(\hat x(p_k,\theta_2),p_k)=0
    \end{aligned}
\label{eq:NI-based-learning}
\end{equation}
where the vector $\nu(p_k,\theta_2)\in\rr^N$ collects the NI function values 
\begin{equation}
\nu_i(p_k,\theta_2) = J_i(\hat x_i(p_k,\theta_2),\hat x_{-i}(p_k,\theta_2),p_{k}) - 
\bar J_i(\hat x_{-i}(p_k,\theta_2),p_k)
\label{eq:NI-term}
\end{equation}
for each agent $i=1,\ldots,N$ at the $k$-th estimated GNE.
Due to the fact that $\bar J_i(x_{-i},p)$ is a solution to the local optimization problem~\eqref{eq:best_response}, the direct NI-based formulation~\eqref{eq:NI-based-learning} requires solving a bilevel optimization problem, subject to the coupling constraints imposed for all $k=1,\ldots,M$, which is computationally expensive.

We avoid bilevel optimization by replacing $\bar J_i(x_{-i},p)$ with the value-function model $\hat J_i(x_{-i},p,\theta_{1i})$, which is obtained by solving the following $N$ learning problems
\begin{equation}
    \theta_{1i}^\star = \arg\min \frac{1}{M}\sum_{k=1}^M\left\|\hat J_i(x_{-i,k},p_k,\theta_{1i})-\bar J_{k,i}\right\|_2^2+r_{1i}(\theta_{1i})
\label{eq:value-function-learning}
\end{equation}
where $\bar J_{k,i}$ is an element of the data set~\eqref{eq:dataset} %(see~\eqref{eq:v-f}) 
and $r_{1i}(\theta_{1i})$ is a regularization term for the parameters of the value-function model, such as $r_{1i}(\theta_{1i}) = \rho_{1i}\|\theta_{1i}\|_2^2 + \tau_{1i}\|\theta_{1i}\|_1$. Accordingly, we define the following {\it approximate} NI function values
\begin{equation}
\begin{aligned}
\hat \nu_i(p_k,\theta_2) =&J_i(\hat x_i(p_k,\theta_2),\hat x_{-i}(p_k,\theta_2),p_{k}) - \\
&\hat J_i(\hat x_{-i}(p_k,\theta_2),p_k,\theta_{1i}^\star)
\end{aligned}
\label{eq:NI-term-approx}
\end{equation}
to replace the original NI function values $\nu_i(p_k,\theta_2)$ in~\eqref{eq:NI-term}.

\subsection{Constraint relaxation}
In order to avoid handling constraints in~\eqref{eq:NI-based-learning}, let us
rephrase them as follows:
\begin{equation}
    \begin{aligned}
    \max_{k\in\KK}&\left(\max_{j=1,\ldots,n_g}(\max(g_j(\hat x(p_k,\theta_2),p_k),0)),\right.\\&\left.\max_{t=1,\ldots,n_h} (|h_t(\hat x(p_k,\theta_2),p_k)|)\right)=0.
    \end{aligned}
\label{eq:max-constraint}
\end{equation}
Then, as suggested in~\cite{Bem26b}, we can penalize the violation of~\eqref{eq:max-constraint}
via the additional loss term (cf.~\cite{Nes05})
\begin{equation}
    \begin{aligned}
    \ell^{\,c}_{\beta,\gamma}(\theta_2)=\frac{\beta}{\gamma}\log\Big(&\sum_{k\in\KK}\Big(\sum_{j=1}^{n_g}
        e^{\gamma\max(g_j(\hat x(p_k,\theta_2),p_k),0)}\\
    &+\sum_{t=1}^{n_h}e^{\gamma |h_t(\hat x(p_k,\theta_2),p_k)|}\Big)
    \end{aligned}
\label{eq:smooth-max}
\end{equation}
where $\beta>0$ is a weighting parameter and $\gamma>0$ is a smoothing parameter, $\beta,\gamma\gg 1$.
Note that, as $\gamma\to\infty$, the term in~\eqref{eq:smooth-max} converges to the left-hand side expression in~\eqref{eq:max-constraint}. 

Taking into account the proposed reformulations of both NI-based loss function and constraints, the GNE learning problem can be formulated as the following unconstrained nonlinear optimization problem:
\begin{equation}
    \begin{aligned}
    \theta_2^\star = \arg\min_{\theta_2} \,& \ell^{\,c}_{\beta,\gamma}(\theta_2) + r_2(\theta_2) + \frac{1}{M}\sum_{k=1}^M \ell_{NI}(\hat \nu(p_k,\theta_2))
    \end{aligned}
    \label{eq:GNE-learning-problem}
\end{equation}
where $r_2(\theta_2)$ is a further regularization term of the GNE model parameters and $\ell_{NI}:\rr^N\to\rr$ is a loss function constructed based on the representation $\hat\nu=(\hat \nu_1,\ldots,\hat \nu_N)$ as defined in~\eqref{eq:NI-term-approx} to approximate the original NI gap function at each data point $k$. Possible choices of $\ell_{NI}$ are:
\begin{subequations}
\begin{equation}
    \ell_{NI}(\hat\nu) = \sum_{i=1}^N \hat \nu_i
    \label{eq:NI-loss-sum}
\end{equation}
similarly to the NI-based loss in~\eqref{eq:NI-based-learning} or
\begin{equation}
    \ell_{NI}(\hat \nu) = \sum_{i=1}^N \max\left(\hat \nu_i,0\right)
\label{eq:NI-loss-max,0}
\end{equation}
or
the following smooth approximation of $\max(\hat\nu_i,0)$ 
\begin{equation}
    \ell_{NI}(\hat \nu) = \frac{1}{2}\sum_{i=1}^N \left(\hat \nu_i + \sqrt{\left(\hat \nu_i\right)^2+\epsilon}\right)
    \label{eq:NI-loss-soft-max,0}
    \end{equation}
\end{subequations}
which is intended to ease optimization, where $\epsilon$ is a smoothing parameter, $0<\epsilon\ll 1$.
Note that~\eqref{eq:NI-loss-max,0} only penalizes the positive parts of the NI approximation function values $\hat \nu_i$, which prevents the nonlinear optimizer of~\eqref{eq:GNE-learning-problem} from attempting to super-optimize the agents' objectives $J_i$ at the surrogate GNEs $\hat x_i(p_k,\theta_2)$ by exploiting imperfect values of $\hat J_i$.

We remark that a solution $\theta_2^\star$ to the learning problem~\eqref{eq:GNE-learning-problem} provides an estimate $\hat x(p,\theta_2^\star)$ even if a GNE does not exist for the parameter value $p$. In this case, the learned model $\hat x(p,\theta_2^\star)$ provides a decision allocation that minimizes the aggregate equilibrium gap and constraint violation, with a trade-off between the two terms that can be tuned via the weighting parameter $\beta$ in~\eqref{eq:smooth-max}.

If any constraint violation is undesired, {\it after} solving the posed learning problem we can cascade the evaluation of the model function with a projection step that enforces such constraints at prediction time, by solving the following optimization problem for a given $p$:
\begin{equation}
\ba{rcrl}
\hat x^{\rm proj}(p,\theta_2^\star) &\in& \arg\min_{x} & \|x-\hat x(p,\theta_2^\star)\|_2^2 \\
&& \st & g(x,p)\leq 0\\
&&& h(x,p)=0.
\label{eq:projection}
\ea
\end{equation}
where $g$, $h$ are the constraint functions in~\eqref{eq:best_response}, 
or a subset of them if only a subset of constraints is to be enforced.
The projection step may be computationally expensive, but should still be significantly faster than solving the original GNEP~\eqref{eq:GNEP-parametric}. 

For the particular case of having only box constraints in~\eqref{eq:best_response}, the projection~\eqref{eq:projection} corresponds to a simple clipping operation. In such a case, an alternative approach to enforce the box constraints 
{\it during} the learning problem is to design the surrogate GNE model $\hat x(p,\theta_2^\star)$ to satisfy the box constraints by adding a saturation function at the model output, such as the smooth one suggested in~\cite[Section 3.B]{Bem25}.

\section{Model Choice in Structured Games}\label{sec:structured}
To approach the posed learning problem, we have implicitly assumed that $\hat J_i(x_{-i}$, $p$, $\theta_{1i})$ is continuous in $x_{-i}$ and $p$, and the GNE model $\hat x(p,\theta_2)$ is continuous in $p$, in particular when feedforward neural networks with continuous activation functions are used as parametric models. In this section, we demonstrate that, under some additional assumptions on the parametric GNEP under consideration, this choice is well justified.

\subsection{Value-function model}
It is well known that the scalar value function $f(p) = \min_{x}\{\phi(x,p)\,|\,\psi(x,p)\le 0\}$ is convex and continuous provided that both functions $\phi: \rr^{n_x}\times\rr^{n_p}$ and $\psi: \rr^{n_x}\times\rr^{n_p}\to\rr^m$ are convex and continuous in $(x,p)$~\cite{MR64}. This result immediately justifies a continuous and convex value function model $\hat J_i(x_{-i}, p,\theta_{1i})$ in $x_{-i}$ and $p$ for each $i$, if $J_i$, $g_i$, and $h_i$ are convex in $(x,p)$. Indeed, according to our construction of the learning problem, see in particular~\eqref{eq:value-function-learning}, $\hat J_i$ should approximate the value function~\eqref{eq:value-function} of the local best-response problem of the agent $i$ defined in~\eqref{eq:best_response}. 

If the joint convexity assumption is not fulfilled, but rather a local convexity holds in terms of cost functions, i.e., each $J_i$ is convex in $x_i$ for any fixed $x_{-i}$ and $p$, the result of Proposition~1 in~\cite{Song2024GeneralizedDerivatives} can be applied. In this case, a continuous model for each $\hat J_i(x_{-i}, p,\theta_{1i})$ in $x_{-i}$ and $p$ is justified provided that:
\begin{enumerate}
    \item [$i$)] the functions $J_i$, $g_i$, and $h_i$ 
    are continuously differentiable over the whole domain;
    \item [$ii$)] the best response optimization
    problem~\eqref{eq:best_response} has a unique optimal solution;
    \item [$iii$)] for each $p\in\mathcal{P}$, the constraint function $g_i$ is convex in $x$, $h_i$ is affine in $x$,  whereas the cost function $J_i$ is convex merely in $x_i$;
    \item [$iv$)] $\bar J_i(x_{-i},p)$ is finite;
    \item [$v$)] for each $p\in\mathcal{P}$, the strong Slater's condition holds for the constraint set in the problem~\eqref{eq:best_response}.
 \end{enumerate}

\subsection{GNE model}
As it has been mentioned above, a GNEP is generally hard to solve, whereas under specific assumptions on the game structure, some useful solution properties can be formulated which further help to make conclusions regarding the form of a parametric solution to~\eqref{eq:GNEP-parametric}. That is why we focus here on a jointly convex GNEP, where the coupling constraints are convex and shared, and each cost function $J_i$ is convex and differentiable in $x_i$ for any fixed $x_{-i}$ and $p$. In this case, the set of GNE solutions contains, in general, a distinguished subset of solutions (possibly empty in pathological cases) called variational generalized Nash equilibria (v-GNE). These equilibria are characterized as the solutions to the variational inequality $VI(F(\cdot,p),X(p))$, where
\[
F(x,p)=\left[\frac{\partial J_1(x,p)}{\partial x_1},\ldots,\frac{\partial J_N(x,p)}{\partial x_N}\right]^\top
\]
is the pseudo-gradient mapping, and
\[
X(p)=\{x:\; g(x,p)\le 0,\; h(x,p)=0\}.
\]
The continuity property of the solution maps to the variational inequality above can be obtained by means of sensitivity analysis, where the behavior of the solution set is studied in a neighborhood of some reference parameter point $p\in\mathcal{P}$.  
Thus, to justify the use of a continuous GNE model in our learning problem, we rely on results from the sensitivity analysis of variational inequalities applied to $VI(F(\cdot,p),X(p))$, see, e.g.,~\cite{Dafermos,FaccPang1,HanPang2024}. In particular, Section~5 of~\cite{FaccPang1} is devoted to the sensitivity analysis of general variational inequalities. The theory developed there is based on topological properties of the so-called natural and normal maps, which are difficult to analyze in practice. By contrast,~\cite{Dafermos} provides sufficient conditions for continuity of the solution map under assumptions including uniform strong monotonicity and Lipschitz continuity of $F(\cdot,p)$ with respect to $x$, for each $p\in\mathcal P$. More recently,~\cite{HanPang2024} studies the case of merely monotone mappings, under the existence of directional derivatives with respect to the parameter and affine parametric constraints. In that setting, the authors establish necessary and sufficient conditions for the existence of a continuous selection of solutions to the corresponding monotone variational inequality. In particular, these conditions require uniqueness of the solution for each $p\in\mathcal P$.

Motivated by the above discussion and to extend the validity of the learning method proposed in the previous sections, here we derive new sufficient conditions for the existence of a continuous selection of parametric v-GNE solutions for the family of games
\[
\Gamma_p=\Gamma\bigl(N,\{J_i(\cdot,p)\},X(p)\bigr), \qquad p\in\mathcal P,
\]
under assumptions that relax monotonicity. Clearly, if a continuous selection of v-GNE solutions exists, then a continuous GNE model $\hat x(p,\theta_2)$ exists as well and can be determined by the learning procedure.

\begin{assumption}\label{assum:strVS}
For each $p\in\mathcal{P}$ the game $\Gamma_p$ is jointly convex and admits a v-GNE solution $x^\star(p)$. Moreover, the pseudo-gradient $F(x,p)$ is continuous in $x$ and $p$ and is strongly variationally stable (SVS) with respect to $x^\star(p)$ with a uniform constant $\nu$, i.e.
\begin{align}\label{eq:strVS}
  \langle F(x,p), x-x^\star(p)\rangle \ge \nu\|x-x^\star(p)\|^2, \,\,\, \forall x\in X(p), p\in\mathcal{P}.  
\end{align}
\end{assumption}

Assumption~\ref{assum:strVS} is more general than the strong monotonicity assumption considered, for example, in~\cite{Dafermos}. Some examples of non-monotone games satisfying Assumption~\ref{assum:strVS} are provided in~\cite{TatKamStrVS}. Any strongly monotone game possesses a strongly variationally stable pseudo-gradient. Finally, note that if Assumption~\ref{assum:strVS} holds, the corresponding variational solution $x^\star(p)$ is necessarily unique in $\Gamma_p$.

\begin{assumption}\label{assum:Proj}
For any fixed $ p\in \mathcal P$ there exist constants $\delta>0$ and $c_p>0$ such that for all $q\in \mathcal P\cap B_\delta(p)$  the following holds:
\[
\bigl\|x^\star(q)-\Pi_{X(p)}(x^\star(q))\bigr\| \le c_p\|p-q\|.
\]
\end{assumption}
The next lemma provides a sufficient condition under which Assumption~\ref{assum:Proj} holds. 
\begin{lemma}\label{lem:Assum}
Let
\begin{enumerate}
    \item[1)]Assumption~\ref{assum:strVS} hold. Moreover, let 
for every $ p\in\mathcal P$, a neighborhood $U_{p}\subset \mathcal P$, a constant $M_y>0$, and a mapping $y:U_{p}\to\mathbb R^n$ exist such that $y(q)\in X(q)$ and  $\|y(q)\|\le M_y$,  for all $ q\in U_{p}$\footnote{This condition holds, in particular, in the case of $X(p)= \{ x: \, A x \le b( p)\}$ with a continuous parametric vector-function $b(p)$.}.
\item[2)]The set $X(p)$ be given by the inequality constraints with functions that are Lipschitz continuous over the set $\mathcal{P}$, i.e. $X(p) = \{x \,:\, g(x,p)\le 0 \}$, $g = (g_1,\ldots, g_m)$,
and  $|g_i(x,p) - g_i(x,q)|\le L_i(x)\|p-q\|$,
where $L_i(x)$ is continuous in $x$. Moreover, let the Slater's constraint qualification condition hold for each $X(p)$.
\end{enumerate}
Then Assumption~\ref{assum:Proj} holds.
\end{lemma}
The proof can be found in the Appendix.

\begin{assumption}\label{assum:Proj1}
For any fixed $y\in\rr^{n_x}$ the mapping $p\to\Pi_{X(p)}{y}$ is continuous over $\mathcal{P}$.
\end{assumption}
The assumption above is a standard regularity condition in sensitivity analysis (see, for example, Remark 2.1 and Proposition 2.1 in~\cite{Dafermos} for some sufficient conditions for Assumption~\ref{assum:Proj1} to hold) that allows for taking limits in the parametric projection operators by establishing continuity of the solution map. 

\begin{proposition}\label{prop:cont}
    If Assumptions~\ref{assum:strVS}-\ref{assum:Proj1} hold, then there exists a continuous parametric v-GNE solution map $x^\star(p)$ to the games $\Gamma_p$, $p\in\mathcal{P}$. 
\end{proposition}
See the Appendix for the proof.

The structural assumptions imposed on the parametric GNEP in this section are sufficient to ensure well-posedness of the learning problem when continuous surrogate models are employed for the value functions $\hat J_i(x_{-i}, p; \theta_{1i})$, $i=1,\ldots,N$, and for the GNE mapping $x^\star(p,\theta_2)$. At the same time, the learning problem can also be formulated without requiring these assumptions to hold. In that case, one may still fit continuous models to the training data by identifying parameters $\theta^\star_{1i}$ and $\theta^\star_2$ through~\eqref{eq:value-function-learning} and~\eqref{eq:GNE-learning-problem}. The resulting learned GNE model then induces a decision map $\hat x(p,\theta_2^\star)$ that is continuous in the parameter $p$ and attempts to minimize the aggregate unilateral improvement over all players. Hence, it can be interpreted as a best approximate equilibrium defined on an ``almost'' feasible decision set.

\section{Standard multiparametric programming}
\label{sec:single-agent}
In the special case $N=1$, problem~\eqref{eq:GNEP-parametric} becomes a standard multiparametric programming problem:
\begin{equation}
\ba{rcrl}
x^\star(p) &\in& \arg\min_{x} & J(x,p) \\
&& \st & g(x,p)\leq 0\\
&&& h(x,p)=0.
\label{eq:single-agent}
\ea
\end{equation}
The proposed learning approach can be applied to this special case as well, which is of interest in its own right, for example in the context of embedded model predictive control~\cite{Bem21c}.
When $N=1$, the NI function values~\eqref{eq:NI-based-learning} simply become the gap value
\begin{equation}
\nu(p_k,\theta_2) = J(\hat x(p_k,\theta_2),p_{k}) - J(x^\star(p_k),p_k)
\label{eq:NI-term-single-agent}
\end{equation}
which is clearly zero when $\hat x(p_k,\theta_2)$ is an optimal solution to~\eqref{eq:single-agent}. Note that the terms $J(x^\star(p_k),p_k)$ do not depend on the trainable parameters.
In case the loss $\ell_{NI}(\hat\nu) = \hat \nu_i$ is used as in~\eqref{eq:NI-loss-sum}, 
we can directly set $\hat \nu(p_k,\theta_2) = J(\hat x(p_k,\theta_2),p_{k})$ in~\eqref{eq:GNE-learning-problem}, without the need to introduce a value-function model $\hat J$ as in~\eqref{eq:NI-term-approx}. Then, the learning problem~\eqref{eq:GNE-learning-problem} simplifies to the following problem
\begin{equation}
    \begin{aligned}
    \theta_2^\star = \arg\min_{\theta_2} \,& \ell^{\,c}_{\beta,\gamma}(\theta_2) + r_2(\theta_2) + \frac{1}{M}\sum_{k=1}^M J(\hat x(p_k,\theta_2),p_{k}).
    \end{aligned}
    \label{eq:GNE-learning-problem-single-agent}
\end{equation}
Problem~\eqref{eq:GNE-learning-problem-single-agent} is equivalent to the problem
of fitting an optimal solution mapping $\hat x(p_k,\theta_2)$ by jointly minimizing the cost function $J$ for all the parameters $p_k$ in the dataset, under the penalty $\ell^{\,c}_{\beta,\gamma}(\theta_2)$ to promote constraint satisfaction and $\ell_2$-regularization of $\theta_2$. Remarkably, this does not require evaluating any explicit optimal values $\bar J_k$ when preparing the training dataset $\DD$, i.e., we simply collect $M$ samples $p_k$ of the parameter vector and set $\DD=\{p_k\}$, $k=1,\ldots,M$. 

Note that the multi-agent problem~\eqref{eq:best_response} is equivalent to $N$ single-agent problems when the cost functions are decoupled, i.e. $J_i(x,p)=J_i(x_i,p)$ for all $i=1,\ldots,N$, and there are no shared constraints, i.e., 
for each given $p$, each component of $g$ and $h$ depends only on one of the agent subvectors $x_i$ for some $i$. In the absence of a game, the learning approach~\eqref{eq:GNE-learning-problem-single-agent} can be applied to each agent's problem separately. 

\section{Numerical Examples}
\label{sec:examples}
In all the examples we collect datasets as follows: ($i$) we generate $M$ parameter values $p_k$ by Latin hypercube sampling~\cite{MBC79} in a hyper-box $\PP$; ($ii$) we sample uniformly at random $M$ values $x^{\rm ref}_k$ within a hyper-box $\XX$, which is defined by parameter-independent lower and upper bounds on $x$ in~\eqref{eq:best_response}, and project $x^{\rm ref}_k$ onto the feasible set defined by the shared constraints for each $p_k$ by solving a constrained least-distance problem to get feasible values $x_k$; ($iii$) we collect optimal value function data $\bar J_{ik}$ for each agent $i$ and pair $(x_k,p_k)$. The resulting dataset $\DD=\{(x_k,p_k,\bar J_k)\}$, $k=1,\ldots,M$, is used to solve the learning problem. Separate validation $\DD_{\rm val}$ of $M_{\rm val}$ samples and test $\DD_{\rm test}$ datasets of $M_{\rm test}$ samples are collected in the same way. We consider only problems without shared equality constraints.

All numerical tests are run in Python 3.11 with the \texttt{mpfit} package on a MacBook Pro with Apple M4 Max (16 CPU cores) and the \texttt{JAX} library for automatic differentiation. Unless differently specified, the training is performed by using the Adam optimizer~\cite{KB14} with a learning rate of $0.001$ for 1000 epochs, followed by 2000 L-BFGS iterations~\cite{LN89}, from 32 different initial guesses. The resulting model is selected as the one giving the lowest objective on $\DD_{\rm val}$ (i.e., the lowest mean-squared error between the predicted and true values of $J_i$ when training the value-function models, and the lowest NI-based loss when training the GNE model), and we report the results on $\DD_{\rm test}$. We use $\gamma=10$ in all tests, as suggested in~\cite{Bem26b}. 

After training, the prediction model $\hat x(p,\theta_2^\star)$ is cascaded by a clipping function enforcing
saturation between the given lower and upper bounds. The quality of the resulting predictions is assessed by the mean-squared best-response errors MSE$_{\rm BR}$ = $\frac{1}{N_{\rm test}}\sum_{k=1}^{N_{\rm test}}\sum_{i=1}^N\|\hat x_i(p_k,\theta_2^\star) - \bar x_i(\hat x_{-i}(p_k,\theta_2^\star),p_k)\|_2^2$, where $\bar x_i(\hat x_{-i}(p_k,\theta_2^\star),p_k)$ is a solution to the best-response problem of agent $i$ defined in~\eqref{eq:best_response} with $x_{-i}=\hat x_{-i}(p_k,\theta_2^\star)$.  
To prevent best-responses from being undefined due to small constraint violations of $\hat x$ that cannot be recovered, all the shared inequality constraints are relaxed into $g(x)\leq [s\ \ldots\ s]^\top$, where $s$ is a slack variable heavily penalized as $\frac{1}{2}\rho s^2+\rho s$ in the best-response optimization problem~\eqref{eq:best_response}, $\rho=10^6$.
Best responses are computed by using the DAQP solver~\cite{ABA22b} for LQ-GNEs, 
CVXPY~\cite{DB16} to formulate quadratically-constrained quadratic programs (QCQP) 
and Clarabel~\cite{GC24} to solve them, and L-BFGS-B~\cite{BLNZ95} via the \texttt{Jaxopt} library~\cite{JAXopt} for nonlinear GNEs when required.

\subsection{Linear-quadratic GNEPs}
We apply the proposed learning method to the following simple 2-agent strongly-monotone 
linear-quadratic GNEP
defined by
\begin{equation}
\ba{rll}
J_{1} &= 0.995 x_{1}^2 - 0.67 x_{2} x_{1} + \left(-0.84 - 1.3 p_{1} + 1.2 p_{2}\right) x_{1} \\
J_{2} &= 0.56 x_{2}^2 + 0.09 x_{1} x_{2} + \left(0.7 + 1.45 p_{1} + 0.09 p_{2}\right) x_{2} \\[.5em]
\st & {\begin{bmatrix}-0.98 & 0.05 \\ 0.16 & -1.21 \\ 2.22 & 0.39 \\ 1.69 & -1.11 \\ 1.64 & -1.36\end{bmatrix}} x \leq {\small\begin{bmatrix}1.27 \\ 0.68 \\ 0.88 \\ 1 \\ 1.19\end{bmatrix}} + {\begin{bmatrix}0.13 & 0.16 \\ 0.16 & 0.19 \\ 0.14 & 0.16 \\ 0.12 & 0.15 \\ 0.12 & 0.19\end{bmatrix}} p\\
& -1\leq x_1,x_2\leq 1
\ea
\label{eq:GNE_LQ}
\end{equation}
that we want to learn over the parameter set $\PP = \{p\in\rr^2:\ -1\leq p_1,p_2\leq 1\}$.
The pseudo-gradient mapping of the game is strongly monotone. 
We collect a training dataset $\DD$ of $M=1000$ samples, $N_{\rm val}=1000$ validation samples, and $N_{\rm test}=1000$ test samples as described above; we solve the learning problem~\eqref{eq:GNE-learning-problem} with the three different NI-based loss functions defined in~\eqref{eq:NI-loss-sum},~\eqref{eq:NI-loss-max,0}, and~\eqref{eq:NI-loss-soft-max,0}, for different values of the weighting parameter $\beta$ in~\eqref{eq:smooth-max}
and $\ell_2$-regularization coefficients equal to $10^{-8}$ in all cases. 
We use feedforward neural networks with 2 hidden layers of 30, 20 neurons, and swish activation function as parametric models for the value functions $\hat J_i$, which are only used for training the GNE model, and 2 hidden layers of 10, 5 neurons, linear bypass, and ReLU activation function for the two-output solution mapping $\hat x$ (used for training and evaluation).  The results are reported in Table~\ref{tab:GNE_LQ} in terms of the average $\bar v$ and worst-case $v_{\rm max}$ constraint violations, and MSE$_{\rm BR}$. Training time ranges between 29.1 and 37.5\,s, with 10.78\,s employed to train the value-function models.

\begin{table}[h!]
\begin{center}
\begin{tabular}{rcccc}
\hline
$\beta$ & loss & $\overline{v}$ & $v_{\max}$ & MSE$_{\rm BR}$ \\
\hline
1 & \eqref{eq:NI-loss-sum} & 3.71e-01 & 7.18e-01 & 7.11e-01 \\
10 & \eqref{eq:NI-loss-sum} & 2.91e-01 & 3.39e-01 & 6.60e-01 \\
100 & \eqref{eq:NI-loss-sum} & 4.14e-06 & 3.05e-03 & 2.50e-02 \\
1000 & \eqref{eq:NI-loss-sum} & 3.29e-06 & 2.11e-03 & 2.32e-02 \\
\hline
1 & \eqref{eq:NI-loss-max,0} & 8.81e-04 & 2.74e-02 & 2.34e-03 \\
10 & \eqref{eq:NI-loss-max,0} & 3.33e-05 & 1.33e-02 & 2.55e-03 \\
100 & \eqref{eq:NI-loss-max,0} & 3.00e-05 & 6.13e-03 & 9.00e-03 \\
1000 & \eqref{eq:NI-loss-max,0} & 2.37e-05 & 6.99e-03 & 4.66e-02 \\
\hline
1 & \eqref{eq:NI-loss-soft-max,0} & 3.93e-03 & 5.94e-02 & 1.59e-03 \\
10 & \eqref{eq:NI-loss-soft-max,0} & 1.65e-03 & 1.50e-02 & \textbf{9.92e-04} \\
100 & \eqref{eq:NI-loss-soft-max,0} & 9.32e-06 & 6.76e-03 & 1.17e-02 \\
1000 & \eqref{eq:NI-loss-soft-max,0} & \textbf{1.49e-06} & \textbf{1.49e-03} & 4.64e-02 \\
\hline
\hline
\end{tabular}
\end{center}
\caption{Linear-quadratic GNEP defined in~\eqref{eq:GNE_LQ}.
}
\label{tab:GNE_LQ}
\end{table}
The table shows that the best results are achieved by using the smooth-max loss~\eqref{eq:NI-loss-soft-max,0},
and that, as expected, $\beta$ trades off feasibility and optimality of the best responses. Note that, as apparent in
the first rows of the table, small values of $\beta$ allow the solution to violate the constraints and,
due to the loss~\eqref{eq:NI-loss-sum}, the solution model becomes super-optimal, with resulting large deviations from (constrained) best responses. From now on, we will use the smooth-max loss~\eqref{eq:NI-loss-soft-max,0} with $\beta=100$ in all the following examples
to compromise between feasibility and optimality. The time to evaluate the solution model $\hat x(p,\theta_2^\star)$ for a given $p$ is approximately $0.08$\,$\mu$s.

Next, we apply the learning method to a more challenging linear-quadratic GNEP with $N$ agents, $N\in\{2,3,4\}$, with $n_i=2$ variables each, $n_p$ parameters ranging in the interval $[-1,1]$, and $m=20N$ shared inequality constraints in addition to upper and lower bounds $-1\leq x\leq 1$. We generate $1000 n_p$ training, validation, and test samples as described above.
The results obtained with neural networks with the same architecture except $15(N+n_p),10(N+n_p)$ neurons in the hidden layers of the value function models and $5(N+n_p),3(N+n_p)$ neurons in the hidden layers of the solution mapping model,
under $\ell_2$-regularization coefficients equal to $10^{-4}$, are reported in Table~\ref{tab:GNE_LQ_N_np}.
Training time ranges between $254.9$ and $904.1$\,s, evaluating the solution model $\hat x(p,\theta_2^\star)$ takes between $0.12$ and $0.18$\,$\mu$s.

\begin{table}[h!]
\setlength{\tabcolsep}{4.5pt}
\begin{center}
\begin{tabular}{rrrcccc}
\hline
$N$ & $n$ & $n_g$ & $n_p$ & $\overline{v}$ & $v_{\max}$ & MSE$_{\rm BR}$ \\
\hline
2 & 4 & 48 & 2 & 9.40e-04 & 8.91e-02 & 6.00e-02 \\
2 & 4 & 48 & 3 & 7.25e-03 & 1.03e-01 & {2.79e-02} \\
2 & 4 & 48 & 4 & 1.23e-02 & 1.69e-01 & 1.26e-01 \\
\hline
3 & 6 & 72 & 2 & {2.12e-04} & 2.71e-02 & 2.88e-02 \\
3 & 6 & 72 & 3 & 4.61e-03 & 1.12e-01 & 6.61e-02 \\
3 & 6 & 72 & 4 & 1.32e-02 & 2.07e-01 & 1.72e-01 \\
\hline
4 & 8 & 96 & 2 & 2.75e-04 & {2.15e-02} & 5.94e-02 \\
4 & 8 & 96 & 3 & 4.61e-03 & 7.19e-02 & 7.43e-02 \\
4 & 8 & 96 & 4 & 1.40e-02 & 1.68e-01 & 3.17e-01 \\
\hline
\end{tabular}
\end{center}
\caption{Random linear-quadratic GNEPs with $N$ agents, $n_i=2$ variables each, 
$20N$ shared inequality constraints, and $n_p$ parameters.}
\label{tab:GNE_LQ_N_np}
\end{table}

\subsection{Non-Monotone Linear-Quadratic GNEP}
To highlight the results of Proposition~\ref{prop:cont}, we consider the following non-monotone linear-quadratic GNEP defined by
\begin{equation}
\ba{rll}
J_{1}(x_{1}, x_{2}, p) &= \frac{1}{2} x_{1}^2 + 2 x_{2} x_{1} + 15 x_{1} \\
J_{2}(x_{1}, x_{2}, p) &= \frac{1}{2} x_{2}^2 + 3 x_{1} x_{2} + \left(3 + 0.4\, p\right) x_{2} \\
\st & x_1+x_2 \leq p - 0.3 \\
& -1\leq x_1,x_2\leq 1,\quad -1\leq p \leq 1.
\ea
\label{eq:GNE_LQ_non-monotone}
\end{equation}
The pseudogradient matrix $G = \left[\begin{smallmatrix}1&2\\3&1\end{smallmatrix}\right]$ has
a symmetric part $G_s = \left[\begin{smallmatrix}1&2.5\\2.5&1\end{smallmatrix}\right]$ with eigenvalues $3.5$ and $-1.5$, so the game is \emph{not} monotone. This game satisfies the assumptions of Proposition~\ref{prop:cont}, which can be checked by inspecting its continuous v-GNE solution map
\[
  x^\star(p) =
  \begin{cases}
    [-1\  0.7+p]^\top & p\in[-1,\,-0.5] \\
    [-1\ -0.4\,p]^\top & p\in[-0.5,\,1].
  \end{cases}
\]
We generate $1000$ training, validation, and test samples as described above. We train value function models with $10,5$ neurons and tanh activation function, solution models with $5,3$ neurons and leaky-ReLU activation function,
and $\ell_2$-regularization coefficients equal to $10^{-8}$.
The resulting solution model achieves an average best-response error of $0.00423$ and no constraint violation on the test set. The total training time is 13.7\,s (9.9\,s to train the value-function models). 

\subsection{Quadratically constrained quadratic GNEPs}
We want to find approximate GNE solutions for the following quadratically constrained quadratic GNEP 
\begin{equation}
    \ba{rll}
x_i &\in \arg\min_{x_i} & 0.5 x^\top Q_i x + c_i^\top x + (F_i p)^\top x \\
&\st & A x \leq b + S p \\
&& 0.5(x-x^c_{j})^\top Q^c_{j}(x-x^c_{j}) \leq b^c_{j} + (s^c_{j})^\top p\\
&& -1\leq x_i \leq 1
\ea
\label{eq:GNE_QCQP}
\end{equation}
with $N=3$ agents, $n_i=1$ variable each, $n_p=2$ parameters
$-1\leq p_1,p_2\leq 1$, shared linear constraints,
defined by $A\in\rr^{8\times 3}$, $b\in\rr^{8}$, and $S\in\rr^{8\times 2}$,
and quadratic constraints, defined by positive definite matrices $Q^c_j\in\rr^{3\times 3}$, 
$x^c_j\in\rr^3$, $b^c_j\in\rr$, and $s^c_j\in\rr^2$ for $j=1,\ldots,4$.
We generate 1000 training, validation, and test samples as described above and train the models with the same architecture and regularization coefficients as for the LQ-GNE~\eqref{eq:GNE_LQ}. 
The resulting solution model $\hat x(p,\theta_2^\star)$ achieves a mean constraint violation of $1.1785\cdot 10^{-4}$ and a worst-case constraint violation of $4.371\cdot 10^{-2}$ on the test set, with an average best-response error of $0.132$. The total training time is 47.6\,s (33.1\,s to train the 3 value-function models). The CPU time to evaluate the resulting GNE prediction $\hat x(p,\theta_2^\star)$ for a given $p$ is approximately 0.114\,$\mu$s. 

\subsection{Nonlinear GNEPs}
We consider a variant of the internet switching model described in~\cite[Ex. A.1]{FK09}:
\begin{equation}
    \begin{aligned}
    &J_i(x,p) = \frac{-x_i}{\sum_{j=1}^N x_j}\left(1-\frac{\sum_{j=1}^N x_j}{p}\right)\\
    &\quad\quad x_1+\ldots+x_N\leq p,\quad x_i \geq \ell, \qquad i=1,\ldots,N
\end{aligned}
\label{eq:A1-FK09}
\end{equation}
where $\ell=0.01$ and $p\in[N\ell,2]$ is treated as the parameter of the game. 
For this problem, we can get the analytical expression for the best response $x_i^\star$ as follows. 
Let $a_i:=\sum_{j\neq i}x_j$ ($a_i>0$ as $\ell>0$), set $J_i(x,p)=-\frac{x_i}{x_i+a_i}+\frac{x_i}{p}$, and get the derivatives
$\frac{\partial J_i}{\partial x_i}=-\frac{a_i}{(x_i+a_i)^2}+\frac{1}{p}$,
and $\frac{\partial^2 J_i}{\partial x_i^2}=\frac{2a_i}{(x_i+a_i)^3}\geq 0$,
which show that $J_i$ is strictly convex in $x_i$ over the feasible set for any fixed feasible $x_{-i}$ and $p$. By zeroing the gradient, we get the unconstrained minimizer
$\widehat{x}_i=\sqrt{pa_i}-a_i$, $\widehat{x}_i\leq p-a_i$ since $a_i<p$. Hence, the best response is given by $x_i^\star=\max\{\ell,\widehat{x}_i\}$, $\forall i=1,\ldots,N$. 

To generate feasible samples, we take 
$p_k$ uniformly at random in $[N\ell,2]$, and set $x_{ki}=\ell + (p_k - N\ell)w_{ki}$, where $w_k\in\rr^{N+1}$ is sampled from the flat Dirichlet distribution ($\sum_{i=1}^{N+1}w_{ki}=1$, $\forall k$).

We consider $N$ agents, $N\in\{2,\ldots,10\}$, and generate 
$M=N_{\rm val}=1000(N-1)$ training and validation samples, and $N_{\rm test}=1000$ test samples as described above.
We train value function models with $(10,5)$ neurons and swish activation function, 
solution models with $(2N, 2N)$ neurons and ReLU activation function,
and $\ell_2$-regularization coefficients equal to $10^{-8}$.
The results are shown in Table~\ref{tab:NL-GNE}.

\begin{table}
\begin{center}
\begin{tabular}{ccccc}
\hline
$N$ & $\overline{v}$ & $v_{\max}$ & $\overline{\epsilon}^2_{\rm BR}$ & training time (s) \\
\hline
2 & 0 & 0 & {6.06e-06} & {17.1} \\
3 & 0 & 0 & 9.11e-05 & 24.8 \\
4 & 0 & 0 & 2.87e-05 & 51.4 \\
5 & 0 & 0 & 9.52e-05 & 81.4 \\
6 & 0 & 0 & 2.01e-04 & 124.4 \\
7 & 0 & 0 & 9.72e-05 & 170.9 \\
8 & 0 & 0 & 1.82e-04 & 234.5 \\
9 & 0 & 0 & 3.34e-04 & 275.5 \\
10 & 0 & 0 & 1.78e-03 & 348.2 \\
\hline
\end{tabular}
\end{center}
\caption{Nonlinear GNEP problem~\eqref{eq:A1-FK09}: results for different numbers $N$ of agents.}
\label{tab:NL-GNE}
\end{table}
The table shows that the proposed learning method is able to find accurate GNE models for this problem, with zero constraint violation and small best-response errors. 

As a second example of nonlinear GNEP, we consider the following problem with $N$ agents, $n_i=2$ variables each, $n_p=2$ parameters, $-1\leq p\leq 1$, with objectives
\begin{equation}
    \begin{aligned}
J_i(x,p) = &(i+1)\Big(\sum_{j=1}^N x_j\Big)^2 + p_1(N-i)\sum_{j\neq i}x_j\sum_{j=1}^N x_j \\
    &+ p_2 \tanh\Big(i\sum_{j=1}^N x_j^2\Big)
    \end{aligned}
\label{eq:nonconvex-GNE}
\end{equation}
subject to bound constraints $-1\leq x\leq 1$. We generate 5000 training, 1000 validation, and 1000 test samples
by evaluating optimal value function data $\bar J_{ik}$ for each agent $i$ and pair $(x_k,p_k)$ by solving a box-constrained nonlinear optimization problem using L-BFGS-B. We train value function models with $(10,5)$ neurons and swish activation function, solution models with $(10,5)$ neurons and ReLU activation function, and $\ell_2$-regularization coefficients equal to $10^{-8}$. The resulting solution model achieves an average best-response error of $9.2313\cdot 10^{-2}$ on the test set. The total training time is 73.9\,s (46.7\,s to train the value-function models). The CPU time to evaluate the resulting GNE prediction $\hat x(p,\theta_2^\star)$ for a given $p$ is approximately 0.43\,$\mu$s.

\subsection{Approximate multiparametric programming}
We consider the special case of finding approximate solutions to multiparametric programming
problems, corresponding to the single-agent case described in Section~\ref{sec:single-agent}.
We test two randomly generated problems: a multiparametric quadratic programming (mpQP) problem with
$n_x=10$ decision variables, $n_p=6$ parameters $p_i\in[0,1]$, $i=1,\ldots,s$, and $m=50$ linear constraints; and a multiparametric quadratically constrained QP (mpQCQP) problem, which does not admit an exact explicit representation,  obtained by adding 20 random convex quadratic constraints to the mpQP.
After collecting 5000 training, validation, and test value-function samples, Problem~\eqref{eq:GNE-learning-problem-single-agent} is solved using a two-layer ReLU network with linear bypass and (30,20) neurons. Such a model is warm-started by minimizing the regularized mean-squared error on optimal solutions available
in the training set. We define the average relative error as
\[
\bar e_{\rm rel}=\frac{1}{N_{\rm test}}\sum_{k=1}^{N_{\rm test}}\frac{J(\hat x(p_k,\theta_2^\star))-J(x^\star(p_k),p_k)}{10^{-10}+|J(x^\star(p_k),p_k)|}.
\]
Results are summarized in Table~\ref{tab:mpp}. For comparison, we also solve the mpQP problem using the multiparametric QP solver~\cite{AA24}, which provides an exact solution with 1627 critical regions.

\begin{table}[h]
\centering
\begin{tabular}{lrr}
\hline
 & mpQP & mpQCQP \\
\hline
Avg.\ relative error $\bar e_{\rm rel}$ & 0.000807 & 0.00218 \\
Avg.\ constraint violation & $0.00306$ & $0.00377$ \\
Worst-case constraint violation & $0.171$ & $0.245$ \\
Training time (s) & 130.3 & 269.7 \\
CPU time ($\hat x$) & $0.16$ $\mu$s & $0.18$ $\mu$s \\
CPU time ($x^\star$) & $0.45$ ms& $2.18$ ms\\
\hline
\end{tabular}
\caption{Results for approximate multiparametric programming.}
\label{tab:mpp}
\end{table}

\section{Conclusions}\label{sec:conclusions}
We proposed a learning-based framework for approximating solution mappings of multiparametric GNEPs and standard multiparametric programming problems, in particular using neural network models of the GNE solution. To justify the use of continuous models, we also established new sufficient conditions for the existence of a continuous parametric v-GNE selection under a strong variational stability assumption. 

A possible extension of our approach to impose feasibility of the predicted solution is to incorporate the projection step~\eqref{eq:projection} during training by using differentiable optimization layers, and explore the benefits of using the obtained approximate GNE models to warm-start GNE solvers.
A future research direction is to further leverage sensitivity properties of variational inequalities to focus on learning specific v-GNE selections.

\appendix

\subsection*{Proof of Lemma~\ref{lem:Assum}}

\begin{proof}
Condition 1) implies that  the solution $x^\star(q)$ is uniformly bounded over $q\in U_p$. Let $X$ be a compact set containing $x^\star(q)$, $q\in U_p$.
Using condition 2), we obtain
\begin{align}\label{eq:Lip}
    g_i(p,x) = g_i(q,x)+\bigl(g_i(p,x)-g_i(q,x)\bigr)\cr
\le
|g_i(p,x)-g_i(q,x)|
\le
L_{i}(x)\|p-q\|.
\end{align} 
Therefore,
\[
[g_i(p,x)]_+ \le L_{i}(x)\|p-q\|,
\qquad i=1,\dots,m.
\]
Next we use the error-bound condition, which holds due to the Slater's qualification one~\cite[Prop. 6.1.4]{FaccPang1}. This condition implies that for each $x\in X$,
\[
\operatorname{dist}(x,X(p)) \le \kappa_{X,p} \sum_{i=1}^m [g_i(p,x)]_+.
\]
Summing~\eqref{eq:Lip} over $i=1,\dots,m$ and using the error-bound condition above, we conclude that for any  $x\in X$, 
\[
\operatorname{dist}(x,X(p)) \le \kappa_{X,p} \sum_{i=1}^m [g_i(p,x)]_+
\le \kappa_{X,p} \sum_{i=1}^m L_{i}(x)\|p-q\|.
\]
Thus,
\[
\operatorname{dist}(x,X(p)) \le c_p\|p-q\|,
\]
where
$c_p = \kappa_{X,p} \sum_{i=1}^m \max_{x\in X}L_{i}(x)$.
The final estimate follows from the identity
\[
\operatorname{dist}(x^\star(q),X(p))
=
\bigl\|x^\star(q)-\Pi_{X(p)}(x^\star(q))\bigr\|
\]
and the fact that $x^\star(q)\in X$.
\end{proof}

\subsection*{Proof of Proposition~\ref{prop:cont}}
\begin{proof}
Let $x^\star(p)$ and $x^\star(q)$ be the solutions to the games $\Gamma_p$ and $\Gamma_q$, $p,q\in\mathcal{P}$, respectively. 
Due to Assumption~\ref{assum:strVS} applied to the game $\Gamma_q$, 
\begin{align}\label{eq:Vst1}
    \langle F(\Pi_{X(q)}[x^\star(p)],q), \,&\Pi_{X(q)}[x^\star(p)]-x^\star(q)\rangle\cr& \ge \nu\|\Pi_{X(q)}[x^\star(p)]-x^\star(q)\|^2.
\end{align}
Moreover, since $x^\star(p)$ is a v-GNE in $\Gamma_p$,
\begin{align}\label{eq:vGNE1}
    \langle F(x^\star(p),p), \,&\Pi_{X(p)}[x^\star(q)]-x^\star(p)\rangle \ge 0.
\end{align}
Taking into account that 
\begin{align*}
   & \langle F(x^\star(p),p), \,\Pi_{X(p)}[x^\star(q)]-x^\star(p)\rangle \cr= &\langle F(x^\star(p),p), \,x^\star(q) - \Pi_{X(q)}[x^\star(p)]\rangle\cr
   &+ \langle F(x^\star(p),p), \, \Pi_{X(p)}[x^\star(q)] - x^\star(q)\rangle\cr
   &+\langle F(x^\star(p),p), \, \Pi_{X(q)}[x^\star(p)] - x^\star(p)\rangle
\end{align*}
and by summing up~\eqref{eq:Vst1} and~\eqref{eq:vGNE1}, we obtain
\begin{align}
   &\nu\|\Pi_{X(q)}[x^\star(p)]-x^\star(q)\|^2\cr
   &\le \langle F(\Pi_{X(q)}[x^\star(p)],q) - F(x^\star(p),p), \,\Pi_{X(q)}[x^\star(p)]-x^\star(q)\rangle\cr
   &\quad+ \langle F(x^\star(p),p), \, \Pi_{X(p)}[x^\star(q)] - x^\star(q)\rangle\cr
   &\quad+\langle F(x^\star(p),p), \, \Pi_{X(q)}[x^\star(p)] - x^\star(p)\rangle\cr
   &\le \|F(x^\star(p),q) - F(x^\star(p),p)\|\|\Pi_{X(q)}[x^\star(p)]-x^\star(q)\|\cr
   &\quad+\|F(x^\star(p),q) - F(\Pi_{X(q)}[x^\star(p)],q)\|\|\Pi_{X(q)}[x^\star(p)]-x^\star(q)\| \cr
   &\quad+ \| F(x^\star(p),p)\| \, \|\Pi_{X(p)}[x^\star(q)] - x^\star(q)\|\cr
   &\quad+\| F(x^\star(p),p)\| \, \|\Pi_{X(q)}[x^\star(p)] - x^\star(p)\|.
\end{align}
 Due to continuity of $F$, Assumptions~\ref{assum:Proj} and~\ref{assum:Proj1}, the last three terms in the inequality above, namely,
 \[\|F(x^\star(p),q) - F(\Pi_{X(q)}[x^\star(p)],q)\| \|\Pi_{X(q)}[x^\star(p)]-x^\star(q)\|\]
 \[ \| F(x^\star(p),p)\| \, \|\Pi_{X(p)}[x^\star(q)] - x^\star(q)\|\] 
 \[\| F(x^\star(p),p)\| \, \|\Pi_{X(q)}[x^\star(p)] - x^\star(p)\|\]
 tend to 0 as $q\to p$.
   Let $\bar x^\star$ be a limit point of $x^\star(q)$ (which exists due to uniform boundedness of $x^\star(q)$) over some subsequence $q_z$. Thus, using again Assumption~\ref{assum:Proj1}, we obtain as $z\to \infty$, 
 \begin{align}\label{eq:limit}
   \nu\|x^\star&(p) - \bar x^\star\|^2 \cr&\hspace*{-1em}\le \lim_{z\to \infty}\|F(x^\star(p),q_z) - F(x^\star(p),p)\|\|x^\star(p)-\bar x^\star\|
\end{align}
which, together with continuity of $F(x,p)$ in $p$, implies that $\bar x^\star = x^\star(p)$. Thus, any limit point of $x^\star(q)$ is $x^\star(p)$. Hence, $\lim_{q\to p} x^\star(q) = x^\star(p)$. 
\end{proof}

\end{document}